\documentclass[12pt]{amsart}

\usepackage{amsmath}
\usepackage{amssymb}
\usepackage{amsfonts}
\usepackage{amscd}
\usepackage{thmdefs}
\usepackage[pagebackref,hypertexnames=false, colorlinks, citecolor=red, linkcolor=red]{hyperref}
\usepackage[backrefs]{amsrefs}
\usepackage[color=green!15,bordercolor=red,linecolor=red!30]{todonotes}
\usepackage{ulem}

\input{tcilatex}
\theoremstyle{definition}
\theoremstyle{remark}
\newtheorem{question}[theorem]{Question}
\numberwithin{equation}{section}

\input tcilatex

\begin{document}
\title{The Linear Bound for Haar Multiplier Paraproducts}

\author[K. Bickel]{Kelly Bickel}
\address{Kelly Bickel, School of Mathematics\\
Georgia Institute of Technology\\
686 Cherry Street\\
Atlanta, GA USA 30332-0160}
\email{kbickel3@math.gatech.edu}

\author[E. T. Sawyer]{Eric T. Sawyer}
\thanks{Research supported in part by a NSERC Grant.}
\address{Eric T. Sawyer, Department of Mathematics, McMaster University, Hamilton, Canada}
\email{sawyer@mcmaster.ca}

\author[B. D. Wick]{Brett D. Wick$^\ddagger$}
\address{Brett D. Wick, School of Mathematics\\
Georgia Institute of Technology\\
686 Cherry Street\\
Atlanta, GA USA 30332-0160}
\email{wick@math.gatech.edu}
\urladdr{www.math.gatech.edu/~wick}
\thanks{$\ddagger$ Research supported in part by National Science Foundation
DMS grant \# 0955432.}

\begin{abstract}
We study the natural resolution of the conjugated Haar multiplier $T_\sigma$:
\begin{equation*}
M_{w^{\frac{1}{2}}}T_{\sigma}M_{w^{-\frac{1}{2}}}=\left( \mathsf{P}_{\widehat{w^{%
\frac{1}{2}}}}^{(0,1)}+\mathsf{P}_{\widehat{w^{\frac{1}{2}}}}^{(1,0)}+%
\mathsf{P}_{\langle w^{\frac{1}{2}}\rangle }^{(0,0)}\right) T_{\sigma}\left( \mathsf{P%
}_{\widehat{w^{-\frac{1}{2}}}}^{(0,1)}+\mathsf{P}_{\widehat{w^{-\frac{1}{2}}}%
}^{(1,0)}+\mathsf{P}_{\langle w^{-\frac{1}{2}}\rangle }^{(0,0)}\right),
\end{equation*}
where each $M_\varphi$ is decomposed into its canonical paraproduct
decomposition. We prove that each constituent 
operator obtained from this resolution has a linear bound on $L^2(\mathbb{R}^d;w)$ in 
terms of the $A_{2}$ characteristic of $w$.
The main tools used are a \textquotedblleft
product formula\textquotedblright\ for Haar coefficients, the Carleson
Embedding Theorem, the linear bound for the square function, and the 
well-known linear bound of $T_{\sigma}$ on $L^2(w).$
\end{abstract}

\maketitle
\tableofcontents

\section{Introduction and Statement of Main Results}

Let $L^2\equiv L^2\left(\mathbb{R}^d\right)$ denote the space of square
integrable functions over $\mathbb{R}^d$. For a weight $w$, i.e., a positive
locally integrable function on $\mathbb{R}^d$, we set $L^2(w)\equiv L^2(%
\mathbb{R}^d;w)$. In particular, we will be interested in $A_2$ weights,
which are defined by: 
\begin{equation*}
\left[ w\right]_{A_2}\equiv\sup_{I} \left\langle w\right\rangle_I
\left\langle w^{-1}\right\rangle_I,
\end{equation*}
where $\left\langle w\right\rangle_I $ denotes the average of $w$ over a
cube $I$.

An operator $T$ is bounded on $L^{2}(w)$ if and only if $M_{w^{\frac{1}{2}%
}}TM_{w^{-\frac{1}{2}}}$ - the conjugation of $T$ by the multiplication
operators $M_{w^{\pm \frac{1}{2}}}$ - is bounded on $L^{2}$. Moreover, the
operator norms are equal:%
\begin{equation*}
\left\Vert T\right\Vert _{L^{2}(w)\rightarrow L^{2}(w)}=\left\Vert M_{w^{%
\frac{1}{2}}}TM_{w^{-\frac{1}{2}}}\right\Vert _{L^{2}\rightarrow L^{2}}.
\end{equation*}%
In the case that $T$ is a dyadic operator adapted to a dyadic grid $\mathcal{%
D}$, it is natural to study weighted norm properties of $T$ by decomposing
the multiplication operators $M_{w^{\pm \frac{1}{2}}}$ into their canonical
paraproduct decompositions relative to the grid $\mathcal{D}$, i.e. 
\begin{eqnarray*}
M_{w^{\pm \frac{1}{2}}}f &=&\mathsf{P}_{\widehat{w^{\pm \frac{1}{2}}}%
}^{(0,1)}f+\mathsf{P}_{\widehat{w^{\pm \frac{1}{2}}}}^{(1,0)}f+\mathsf{P}%
_{\langle w^{\pm \frac{1}{2}}\rangle }^{(0,0)}f \\
\end{eqnarray*}%
(the paraproduct operators are defined in the next section) and then
decomposing $M_{w^{\frac{1}{2}}}TM_{w^{-\frac{1}{2}}}$ into the nine
canonical individual paraproduct composition operators:%
\begin{eqnarray}
M_{w^{\frac{1}{2}}}TM_{w^{-\frac{1}{2}}} &=&\left( \mathsf{P}_{\widehat{w^{%
\frac{1}{2}}}}^{(0,1)}+\mathsf{P}_{\widehat{w^{\frac{1}{2}}}}^{(1,0)}+%
\mathsf{P}_{\langle w^{\frac{1}{2}}\rangle }^{(0,0)}\right) T\left( \mathsf{P%
}_{\widehat{w^{-\frac{1}{2}}}}^{(0,1)}+\mathsf{P}_{\widehat{w^{-\frac{1}{2}}}%
}^{(1,0)}+\mathsf{P}_{\langle w^{-\frac{1}{2}}\rangle }^{(0,0)}\right)
\label{canoncialpara} \\
&\equiv &Q_{T,w}^{\left( 0,1\right) ,\left( 0,1\right) }+Q_{T,w}^{\left(
0,1\right) ,\left( 1,0\right) }+Q_{T,w}^{\left( 0,1\right) ,\left(
0,0\right) }  \notag \\
&&+Q_{T,w}^{\left( 1,0\right) ,\left( 0,1\right) }+Q_{T,w}^{\left(
1,0\right) ,\left( 1,0\right) }+Q_{T,w}^{\left( 1,0\right) ,\left(
0,0\right) }  \notag \\
&&+Q_{T,w}^{\left( 0,0\right) ,\left( 0,1\right) }+Q_{T,w}^{\left(
0,0\right) ,\left( 1,0\right) }+Q_{T,w}^{\left( 0,0\right) ,\left(
0,0\right) },  \notag
\end{eqnarray}%
where $Q_{T,w}^{\left( \varepsilon _{1},\varepsilon _{2}\right) ,\left(
\varepsilon _{3},\varepsilon _{4}\right) }$ is defined in the obvious way from the expression above.  
If one could show that the operator norms of the $Q_{T,w}^{\left(
\varepsilon _{1},\varepsilon _{2}\right) ,\left( \varepsilon
_{3},\varepsilon _{4}\right) }$ are linear in the $A_{2}$ characteristic,%
\begin{equation*}
\left\Vert Q_{T,w}^{\left( \varepsilon _{1},\varepsilon _{2}\right) ,\left(
\varepsilon _{3},\varepsilon _{4}\right) }\right\Vert _{L^{2}\rightarrow
L^{2}}\lesssim \left[ w\right] _{A_{2}}
\end{equation*}%
it then becomes reasonable to expect that the canonical decomposition of a
dyadic operator $T$ into its paraproduct compositions $Q_{T,w}^{\left(
\varepsilon _{1},\varepsilon _{2}\right) ,\left( \varepsilon
_{3},\varepsilon _{4}\right) }$ will inherit the salient properties of $T$
without losing anything of importance. Of course, these dyadic
paraproduct compositions can be expected to yield to structured dyadic proof
strategies.

This idea has been successfully used to study decompositions of the Hilbert transform in \cite{PRSW}.  
We now extend this idea to the martingale transforms.  Specifically, let $\{\sigma_I\}_{I \in \mathcal{D}}$
denote a sequence of $2^d-1 \times 2^d-1$ diagonal matrices indexed by the dyadic cubes with diagonal entries denoted by
$(\sigma_I)_{\alpha \alpha} \equiv \sigma_{I, \alpha}$ for $\alpha = 1, \dots, 2^d-1.$ Define
\[ 
T_{\sigma} f \equiv \sum_{I \in \mathcal{D}} \left( \sigma_I
 \widehat{f}(I) \right) \cdot  h_I  \qquad \forall f \in L^2,
\]
where $h_I$ 
is the vector of Haar functions adapted to the cube $I$  and $\widehat{f}(I)$
is the vector of Haar coefficients associated to the function $f$. For precise definitions of these Haar objects, see Section \ref{sec:basics} and
 for a precise definition of $T_{\sigma},$ see the beginning of Section \ref{sec:Haar}.
It is well known and simple to see that 
\[
\| T_{\sigma} \|_{L^2 \rightarrow L^2} \le \| \sigma \|_{\infty},
\]
where 
\[ \| \sigma \|_{\infty} \equiv \sup_{I\in\mathcal{D}} \sup_{1 \le \alpha \le 2^d-1} \left| \sigma_{I, \alpha} \right|. \] 
A similar norm bound holds for $L^2(w)$.  
Specifically, in \cite{wit00}, J. Wittwer established the following result in one-dimension and
using related arguments, D. Chung obtained the $d$-dimensional analogue in \cite{C}.

\begin{theorem}[Linear Bound for Martingale Transforms] \label{thm:haarbd}
\[
\| T_{\sigma} \|_{L^2(w) \rightarrow L^2(w)} \lesssim [w]_{A_2} \| \sigma \|_{\infty}.
\] 
\end{theorem}
Wittwer only established the result for the case where each
$\sigma_I \in \{\pm 1\}$, but the general case follows using the same
arguments. F. Nazarov, S. Treil, and A. Volberg obtained more general
results in \cites{ntv99, ntv08}, where they showed that certain testing conditions are 
sufficient to determine when Haar multipliers and related
operators are bounded from $L^2(w)$ to $L^2(v).$

Here, we study the paraproduct decomposition of $T_{\sigma}$ and establish the 
following result:

\begin{theorem} 
\label{thm:haarmultiplierbound}
let $\{\sigma_I\}_{I \in \mathcal{D}}$
denote a sequence of $2^d-1\times2^d-1$ diagonal matrices indexed by the dyadic cubes and let 
$w$ be an $A_2$ weight.  Then, each paraproduct composition in the canonical 
resolution of $T_\sigma$  can be controlled by a 
linear power of $\left[ w\right]_{A_2}$, i.e. 
\begin{equation*}
\left\Vert Q_{T_\sigma,w}^{\left( \varepsilon _{1},\varepsilon _{2}\right) ,\left(
\varepsilon _{3},\varepsilon _{4}\right) }\right\Vert _{L^{2}\rightarrow
L^{2}}\lesssim \left[ w\right] _{A_{2}}\left\Vert \sigma\right\Vert_{\infty},
\end{equation*}
for each $Q_{T_\sigma,w}^{\left( \varepsilon _{1},\varepsilon _{2}\right) ,\left(
\varepsilon _{3},\varepsilon _{4}\right)}$ in \eqref{canoncialpara}.
\end{theorem}
The proof relies heavily on arguments appearing in \cite{PRSW}, especially
the use of a ``product formula" for Haar coefficients, the Carleson
Embedding Theorem, and the linear bound for the square function.  Crucial to the approach taken in this paper is the introduction of Wilson's Haar basis in $\mathbb{R}^n$, \cite{W}, and certain modifications of Chung obtained in \cite{C}.  The approach carried out in \cite{PRSW} is dependent upon the Haar system in $\mathbb{R}$ and the modifications in \cites{C,W} are crucial for our analysis in this note.
To handle the final resolvent paraproduct, we must rely on Theorem \ref{thm:haarbd}. 
Obtaining the bound independent of Theorem \ref{thm:haarbd} is currently an
open question.

\section{Notation and Useful Facts}
\label{sec:basics}

Before proving our main result, we collect necessary notation and estimates.  Throughout this
paper $A \equiv B$ means that the expressions are equal by definition, and $A\lesssim B $ means that there
exists a constant $c_d$, which may depend on the dimension $d$, such that $A\leq c_d B$.

Let $\mathcal{D}$ denote the usual dyadic grid of cubes in $\mathbb{R}
^d$. For $I\in\mathcal{D},$  let $\mathcal{C}_1(I)$
denote the $2^{d}$ children of $I$. Note that each child $J\in\mathcal{C}%
_1(I)$ satisfies $\left\vert J\right\vert=2^{-d}\left\vert I\right\vert$. Further, for $d \in \mathbb{N}$,
set
\begin{equation*}
\Gamma _{d}\equiv \left\{ 0,1\right\} ^{d}\setminus \left\{ \left(
1,\ldots,1\right) \right\} ,
\end{equation*}
and fix an enumeration of this set for the rest of the paper. Elements of this set will be denoted by lowercase greek letters.

\subsection{Wilson's Haar System}

While we would like to use the standard Haar system in the analysis below,
it is more convenient to use an orthonormal system developed by M. Wilson in \cite{W}.  
To construct it, we need the following lemma. It is worth mentioning that property (i) did not 
appear in Wilson's original lemma but was added by D. Chung in \cite{C}.

\begin{lemma}[Wilson, \cite{W}*{Lemma 2} ]
Let $I\in\mathcal{D}$.  Then there are $2^d-1$ pairs of sets $\left\{(E_{\alpha,I}^1,E_{\alpha,I}^2)\right\}_{\alpha\in\Gamma_d}$ such that
\begin{itemize}
\item[(i)] For each $\alpha\in\Gamma_d$, $\left\vert E_{\alpha,I}^1\right\vert=\left\vert E_{\alpha,I}^2\right\vert$;
\item[(ii)] For each $\alpha$ and $s=1,2$, $E_{\alpha,I}^s$ is a non-empty union of cubes from $\mathcal{C}_1(I)$;
\item[(iii)] For each $\alpha$, $ E_{\alpha,I}^1\cap E_{\alpha,I}^2=\emptyset$;
\item[(iv)] For every $\alpha\neq\beta$ one of the following must hold:
\begin{itemize}
\item[(a)] $E_{\alpha,I}^1\cup E_{\alpha,I}^2$ is entirely contained in either $E_{\beta,I}^1$ or $E_{\beta,I}^2$;
\item[(b)] $E_{\beta,I}^1\cup E_{\beta,I}^2$ is entirely contained in either $E_{\alpha,I}^1$ or $E_{\alpha,I}^2$;
\item[(c)] $\left(E_{\beta,I}^1\cup E_{\beta,I}^2\right)\cap \left(E_{\alpha,I}^1\cup E_{\alpha,I}^2\right)=\emptyset$.
\end{itemize} 
\end{itemize}
\end{lemma}
Set $E_{\alpha,I}\equiv E_{\alpha,I}^1\cup E_{\alpha,I}^2$.  It is important to observe that $\left\vert E_{\alpha,I}\right\vert\approx\left\vert I\right\vert$ for some dimensional constants.
Further, given any $E_{\alpha,I}$ and $E_{\beta, J}$, it follows from the properties of $\mathcal{D}$ and (iv) that one of the following must hold:
$E_{\alpha,I} \subsetneq E_{\beta, J}$, $E_{\beta, J} \subsetneq E_{\alpha,I}$, or $E_{\alpha,I} = E_{\beta, J}$.
Given this collection of sets, we can now introduce Wilson's Haar system for $L^2(w)$.  Fix $\alpha\in\Gamma_d$, $I \in \mathcal{D}$ and define:
\begin{equation}
\label{WilsonHaar}
h_I^{w, \alpha}\equiv\frac{1}{\sqrt{w(E_{\alpha,I})}}\left[\frac{\sqrt{w(E_{\alpha,I}^1)}}{\sqrt{w(E_{\alpha,I}^2)}} \mathsf{1}_{E_{\alpha,I}^2}-\frac{\sqrt{w(E_{\alpha,I}^2)}}{\sqrt{w(E_{\alpha,I}^1)}} \mathsf{1}_{E_{\alpha,I}^1}\right].
\end{equation}
It is easy to show that $\{h_I^{w, \alpha}\}_{\alpha \in \Gamma_d, I\in \mathcal{D}}$ is an orthonormal system in $L^2(w)$.  When the weight $w\equiv 1$, we denote this collection of functions by $\{h_I^{\alpha}\}_{\alpha\in\Gamma_d, I\in\mathcal{D}}$.   Note that each function $h_I^{\alpha}$ has a fixed sign on each child cube $I'\in\mathcal{C}_1(I)$.
Now, for a fixed dyadic cube $J$, set 
\begin{eqnarray*}
\widehat{f}\left( J,\alpha\right) &\equiv & \left\langle f,h_{J}^{\alpha}\right\rangle
_{L^2 }\quad\forall \alpha\in\Gamma_d\ , \\
\widehat{f}\left( J\right) & \equiv &\left( \widehat{f}\left( J,\alpha\right)
\right) _{\alpha\in \Gamma _{d}}\ , \\
h_{J} & \equiv &\left( h_{J}^{\alpha}\right) _{\alpha\in \Gamma _{d}}\ , \\
\bigtriangleup _{J} f & \equiv & \widehat{f}\left( J\right) \cdot h_{J}
=\sum_{\alpha\in \Gamma _{d}}\widehat{f}\left( J,\alpha\right) h_{J}^{\alpha}\ .
\end{eqnarray*}
This means that $\widehat{f}(J)$ is the vector of Haar coefficients of the function $f$,
and $h_J$ is the vector of Haar functions. It is easy to see that the set $\left\{h_I^{\alpha}\right\}_{I\in\mathcal{D}, \alpha\in\Gamma_d}$ is an
orthonormal basis for $L^2$ and so, $f=\sum_{I\in\mathcal{D}} \bigtriangleup_{I} f=\sum_{I\in\mathcal{D}} \widehat{f}(I)\cdot h_I$. This implies
\begin{equation*}
\left\Vert f\right\Vert_{L^2}^2=\sum_{I\in\mathcal{D}}\left\vert \widehat{f}%
(I)\right\vert^2=\sum_{I\in\mathcal{D}}
\sum_{\alpha\in\Gamma_d}\left\vert \widehat{f}(I,\alpha)\right\vert^2 
= \sum_{\alpha\in\Gamma_d}\sum_{I\in\mathcal{D}}\left\vert 
\widehat{f}(I,\alpha)\right\vert^2.
\end{equation*}
Now given a set $E$, define
$h_{E}^{1}\equiv\frac{\mathsf{1}_E}{\left\vert E\right\vert}$, so that the function is $L^1$ normalized.  
We also set $\left\langle f\right\rangle_E \equiv \left\langle f, h_E^{1}\right\rangle_{L^2}$.
The Wilson Haar system has the standard martingale property that the average of $f$ over each $E_{\alpha, I}$ satisfies
\[ \left \langle f \right \rangle_{E_{\alpha, I}} = \sum_{J \in \mathcal{D}} \sum_{\beta \in \Gamma_d} \widehat{f}(J, \beta) \left \langle h^{\beta}_J, h^1_{E_{\alpha, I}} \right \rangle_{L^2} 
= \sum_{J: J \supseteq I} \sum_{\beta: E_{J, \beta} \supsetneq E_{\alpha, I}} 
\widehat{f}(J, \beta) \left \langle h^{\beta}_J, h^1_{E_{\alpha, I}} \right \rangle_{L^2}.
\]
A fundamental tool in our study will be the product formula for Haar expansions in $L^2$. 
A version of this previously appeared in \cite{SSUT}.  
Specifically, given two functions $f$ and $g$ in $L^2,$ we  can expand them with respect to this Haar basis and formally obtain
\begin{eqnarray*}
fg & = & \left(  \sum_{I \in \mathcal{D}} \sum_{\alpha \in \Gamma_d} \widehat{f}(I, \alpha) h_I^{\alpha} \right) \times
\left(  \sum_{J \in \mathcal{D}} \sum_{\beta \in \Gamma_d} \widehat{g}(J, \beta) h_J^{\beta} \right)\\
& = & \sum_{\alpha\in\Gamma_d}\sum_{I\in\mathcal{D}}\left\langle g\right\rangle_{E_{\alpha,I}} \widehat{f}%
(I,\alpha)h_I^{\alpha} + \sum_{\alpha\in\Gamma_d}\sum_{I\in\mathcal{D}}\left\langle f\right\rangle_{E_{\alpha,I}} 
\widehat{g}(I,\alpha)h_I^{\alpha}+\sum_{\alpha\in\Gamma_d}\sum_{I\in\mathcal{D}} \widehat{g}(I,\alpha)\widehat{f}(I,\alpha)h_{E_{\alpha,I}}^1,
\end{eqnarray*}
where we use the fact that 
\[ h^{\alpha}_I h^{\beta}_J = \left \langle h^{\beta}_J, h^1_{E_{\alpha, I}} \right \rangle_{L^2} h^{\alpha}_I \text{ whenever } E_{\alpha, I} \subsetneq E_{\beta, J}. \]
Although the product formula above does not necessarily make sense for arbitrary $f,g \in L^2$, it is well-defined if $f, g$ are finite linear combinations of Haar functions. Moreover, for $J\in\mathcal{D}$ and $\beta\in\Gamma_d$, we have
\begin{equation}
\label{product formula}
\widehat{fg}\left( J,\beta\right) = \sum_{\alpha\in\Gamma_d}\sum_{I\in\mathcal{D}}
\widehat{f}(I,\alpha) \widehat{g}(I,\alpha)  \left\langle h_{E_{\alpha,I}}^1, h_{J}^{\beta}\right\rangle_{L^2}+\widehat{f}\left( J,\beta \right) 
\left\langle g\right\rangle_{E_{\beta,J}}+\widehat{g}\left( J,\beta\right) \left\langle f\right\rangle_{E_{\beta,J}}.
\end{equation}
For finite linear combinations  of Haar functions, this is obtained by simply calculating the Haar coefficient corresponding to $\beta\in\Gamma_d$ and $J\in\mathcal{D}$ in formula for the product $fg$. If $f,g$ are locally in $L^2$, we can approximate them on $E_{\beta,J}$ by finite linear combinations of Haar functions and still deduce (\ref{product formula}). 
Primarily, we will use version (\ref{product formula}) of the product formula. However, it should be noted that support conditions in the first
term actually imply that
\begin{equation*}
\widehat{fg}\left( J,\beta\right) = \sum_{I: I \subseteq J} \sum_{\alpha: E_{\alpha, I} \subsetneq E_{\beta, J}}
\widehat{f}(I,\alpha) \widehat{g}(I,\alpha)  \left\langle h_{E_{\alpha, I}}^1, h_{J}^{\beta}\right\rangle_{L^2}+\widehat{f}\left( J,\beta \right) 
\left\langle g\right\rangle_{E_{\beta,J}}+\widehat{g}\left( J,\beta\right) \left\langle f\right\rangle_{E_{\beta,J}}.
\end{equation*}
Motivated by these product decompositions, we consider the following dyadic operators. They will
be of fundamental importance in this paper.  Give a sequence of numbers $a=\{a_{I,\alpha}\}_{I\in\mathcal{D}, \alpha\in\Gamma_d}$ indexed 
by $I\in\mathcal{D}$ and $\alpha\in\Gamma_d$, 
we define the following paraproduct type operators:
\begin{eqnarray*}
\mathsf{P}_{a}^{(0,0)}f & \equiv & \sum_{\alpha\in\Gamma_d} \sum_{I\in\mathcal{D}}a_{I,\alpha} \, \widehat{f}(I,\alpha) h_I^{\alpha} \\
\mathsf{P}_{a}^{(0,1)}f & \equiv & \sum_{\alpha\in\Gamma_d} \sum_{I\in\mathcal{D}} a_{I,\alpha} \left\langle
f\right\rangle_{E_{\alpha,I}} h_I^{\alpha} \\
\mathsf{P}_{a}^{(1,0)}f & \equiv & \sum_{\alpha\in\Gamma_d}\sum_{I\in\mathcal{D}} a_{I,\alpha} \widehat{f}(I,\alpha) h_{E_{\alpha,I}}^{1}.
\end{eqnarray*}
It is easy to see that the operator $M_g$ of multiplication by $g$ can formally be written as 
\begin{equation} \label{Multiplication} M_g f = \mathsf{P}_{\left \langle g \right \rangle }^{(0,0)}f  + \mathsf{P}_{ \widehat{g} }^{(0,1)}f  + \mathsf{P}_{\widehat{g}}^{(1,0)}f \end{equation}
where $\left \langle g \right \rangle \equiv \{ \left \langle g \right \rangle_{E_{\alpha, I} }\}_{I \in \mathcal{D}, \alpha \in \Gamma_d}$ 
and $\widehat{g} \equiv \{ \widehat{g}(I, \alpha) \}_{I \in \mathcal{D}, \alpha \in \Gamma_d}.$ We will use \eqref{Multiplication} to decompose the operators $M_{w^{\pm\frac{1}{2}}}.$

\subsubsection{Disbalanced Haars}

At points in our later arguments, we will use
disbalanced Haar functions. To do so, we require some additional notation.
Fixing a dyadic cube $J$, a weight $w$ on $\mathbb{R}^d,$
and $\beta \in\Gamma_d,$ we set 
\begin{equation}  \label{disbalanced_defs1}
C_J(w,\beta)\equiv \sqrt{\frac{\left \langle w \right \rangle_{E^1_{\beta,J}} 
\left \langle w \right \rangle_{E^2_{\beta,J}}}{\left \langle w \right \rangle_{E_{\beta,J}}}}
\quad \text{
and }\quad D_J(w, \beta)\equiv\frac{\widehat{w}(J,\beta)}{\left\langle w\right\rangle_{E_{\beta, J}}}.
\end{equation}
Then we have 
\begin{equation}  \label{disbalanced_defs2}
h_J^{\beta}=C_J(w,\beta) h_J^{w,\beta}+D_J(w,\beta)h_{E_{\beta,J}}^{1}
\end{equation}
where $\left\{h_I^{w, \alpha}\right\}_{I\in\mathcal{D},\alpha \in\Gamma_d}$ is the $
L^2(w)$ orthonormal system defined in \eqref{WilsonHaar}. To see this, we use the two equations 
\begin{equation*}
\int_J h_J^{w, \beta}w=0\quad \text{ and }\quad \int_{J}
\left(h_J^{w,\beta}\right)^2w=1
\end{equation*}
to solve for $C_J(w,\beta)$ and $D_J(w,\beta)$. The claimed formula for $%
D_J(w,\beta)$ follows immediately from the condition that $h_J^{w,\beta}$ have
integral zero. Using the second condition and the formula for $D_J(w,\beta),$
one can easily prove that 
\begin{equation*}
C_J(w,\beta)=\sqrt{\frac{\left\langle w\right\rangle_{E_{\beta,J}}^2-\left\vert
E_{\beta,J} \right\vert^{-1}\widehat{w}(J,\beta)^2}{\left\langle w\right\rangle_{E_{\beta,J}}}%
}.
\end{equation*}
Using basic manipulations, and the fact that $|E_{\beta,J}| = 2 |E^j_{\beta,J}|$ for $j=1,2$, we have
\[
\begin{aligned}
\left\langle w\right\rangle_{E_{\beta,J}}^2-\left\vert
E_{\beta,J} \right\vert^{-1}\widehat{w}(J,\beta)^2 & =
\frac{1}{|E_{\beta,J}|^2} \left( \left( w(E^1_{\beta,J}) +  w(E^2_{\beta,J}) \right)^2 -  
\left( w(E^2_{\beta,J}) -  w(E^1_{\beta,J}) \right)^2 \right) \\
& = \frac{4}{|E_{\beta,J}|^2}  w(E^1_{\beta,J}) w(E^2_{\beta,J}) 
= \left \langle w \right \rangle_{E^1_{\beta,J}}  \left \langle w \right \rangle_{E^2_{\beta,J}}. 
\end{aligned}
\]
This gives the desired formula for $C_J(w,\beta)$. A useful observation is
\begin{equation}  \label{C_Kest}
C_J(w,\beta)^2\leq 4 \langle w \rangle_{E_{\beta, J}} \leq 2^{d+1} \left \langle 
w \right \rangle_J,
\end{equation}
which follows since each $\left \langle w \right 
\rangle_{E^j_{\beta,J}} \leq 2 \left \langle w \right \rangle_{E_{\beta,J}}$ and as
$E_{\beta,J}$ contains at least two children of $J$,
 $ \left \langle w \right \rangle_{E_{\beta,J}} \leq 2^{d-1} \left \langle w \right \rangle_{J}.$

\subsubsection{Carleson Embedding Theorem}
A major tool in this paper is the following modification of the standard Carleson 
Embedding Theorem to the sets $\{E_{\alpha, I}\}_{I \in \mathcal{D}, \alpha \in \Gamma_d}.$
It appears in \cite{C}*{Theorem 4.3}:

\begin{theorem}[Modified Carleson Embedding Theorem]
Let $w$ be a weight on $\mathbb{R}^d$ and let $\{a_{\alpha, I} \}_{I \in \mathcal{D}, \alpha \in \Gamma_d}$ be a 
sequence of nonnegative numbers. Then, there is a constant $A > 0$ such that
\[
\frac{1}{|E_{\alpha, I}|} \sum_{J \subset I} \sum_{\beta: E_{\beta, J} \subset E_{\alpha, I}} a_{\beta, J} \left
\langle w \right \rangle^2_{E_{\beta, J}}  \le A \left \langle w \right \rangle_{E_{\alpha, I}} \qquad \forall I \in \mathcal{D}, \ \alpha \in \Gamma_d,
\]
if and only if
\[ \sum_{I \in \mathcal{D}} \sum_{ \alpha \in \Gamma_d} a_{\alpha, I} \left \langle w^{\frac{1}{2}}f \right \rangle^2_{E_{\alpha,I}} \lesssim A \| f\|^2_{L^2} \quad \forall  f \in L^2.
\]
\end{theorem}
The proof strategy employed in \cite{C} to deduce this is based on the Bellman technique; however there are other methods by which the interested reader can arrive at the Theorem above.

\subsection{Square Function Estimates}

\subsubsection{Square Function Bound}

Define the dyadic square function $S$ on $L^2$ by 
\[
 Sf(x)^2  = \sum_{I \in \mathcal{D}} | \widehat{f}(I) |^2 h^1_I(x) = \sum_{I \in \mathcal{D}} \sum_{\alpha \in \Gamma_d} | \widehat{f}(I, \alpha) |^2 h^1_I(x) . \]
It is clear from the definition that $\left\Vert Sf \right\Vert_{L^2}=\left\Vert f\right\Vert_{L^2}$.  
Versions of the square function have been studied in the weighted setting $L^2(w)$ 
and it has been shown that a linear bound in terms of the $A_2$ characteristic holds.  
We point the interested readers to \cites{CUMP}.  In the one dimensional case we refer the readers to \cites{wit00,MR1771755}.

For our needs, we require a slightly different formulation and so provide an 
alternate proof of this fact.  Using the arguments from Petermichl and Pott \cite{PetermichlPott}, we prove

\begin{theorem} \label{thm:SquareEst1} Let $w$ be an $A_2$ weight in $\mathbb{R}^d$. Then
\[
\| S f \|_{L^2(w)} \lesssim [w]_{A_2} \| f \|_{L^2(w)} \quad \forall f \in L^2(w). 
\] 
\end{theorem}

\begin{proof}
As in \cite{PetermichlPott}, without loss of generality, we can assume $w$ and $w^{-1}$ are bounded
so long as the bounds do not appear in our final estimates. 
We first prove the lower bound:
\begin{equation}
 \| f \|^2_{L^2(w)} \lesssim [w]_{A_2} \| S f \|_{L^2(w)}^2 \qquad \forall f \in L^2(w).  \label{thm:SquareEst2} 
\end{equation}
To this end, define the discrete multiplication operator $D_w: L^2 \rightarrow L^2$ by
\[ 
D_w:  h^{\alpha}_I \mapsto \langle w \rangle_I h^{\alpha}_I  \quad \forall I \in
 \mathcal{D}, \alpha \in \Gamma_d \]
and let $M_w$ denote multiplication by $w$. 
Then, we can rewrite (\ref{thm:SquareEst2}) as: 
\begin{equation} \label{eqn:est1}
\langle M_w f ,f \rangle_{L^2} \lesssim  [ w  ]_{A_2} 
\langle D_w f ,f \rangle_{L^2}, \qquad \forall f \in L^2.
\end{equation}
First, since $w$ and $w^{-1}$
are bounded, $D_w$ and $M_w$ are bounded and invertible with 
$M_w^{-1} = M_{w^{-1}}$ and $D_w^{-1}$ defined by
\[
D^{-1}_w:  h^{\alpha}_I  \mapsto \langle w \rangle_I^{-1} h^{\alpha}_I \quad \forall I \in
 \mathcal{D}, \alpha \in \Gamma_d.\]
As in \cite{PetermichlPott}, one can convert \eqref{eqn:est1} to the equivalent inverse
inequality:
\begin{equation} \label{eqn:est2}
\langle D^{-1}_w f ,f \rangle_{L^2} \lesssim  [ w  ]_{A_2} 
\langle M^{-1}_w f ,f \rangle_{L^2}, \qquad \forall f \in L^2.
\end{equation}
So, we need to establish:
\begin{equation} \label{eqn:ReverseSquare}
\sum_{I \in \mathcal{D}}  \langle w \rangle^{-1}_I 
\left | \widehat{f}(I) \right |^2 \lesssim   [w]_{A_2} \| f \|^2_{L^2(w^{-1})} \quad \forall f \in L^2.
\end{equation}
As in \cite{PetermichlPott}, our first step is to rewrite the sums using disbalanced Haar
functions adapted to $w$ using \eqref{disbalanced_defs1}. To do so, fix a cube $J$ and $\alpha \in\Gamma_d$ and recall that
\begin{equation*}
C_J(w,\alpha)\equiv \sqrt{\frac{\left \langle w \right \rangle_{E^1_{\alpha,J}} 
\left \langle w \right \rangle_{E^2_{\alpha,J}}}{\left \langle w \right \rangle_{E_{\alpha,J}}}}
\quad \text{
and }\quad D_J(w,\alpha)\equiv\frac{\widehat{w}(J,\alpha)}{\left\langle w \right\rangle_{E_{\alpha, J}}}.
\end{equation*}
Then we have 
\begin{equation*}
h_J^{\alpha}=C_J(w,\alpha) h_J^{w,\alpha}+D_J(w,\alpha)h_{E_{\alpha,J}}^1,
\end{equation*}
where $\left\{h_J^{w,\alpha}\right\}_{J\in\mathcal{D},\alpha\in\Gamma_d}$ is the previously-defined $%
L^2(w)$ orthonormal system. 
Returning to the sum in question, we use the disbalanced Haar functions to write:
\[
\begin{aligned}
\sum_{I \in \mathcal{D}}  \langle w \rangle^{-1}_I 
\left | \widehat{f}(I) \right |^2
&= \sum_{I \in \mathcal{D}} \sum_{\alpha \in \Gamma_d} \langle w \rangle^{-1}_I 
 \left \langle f, h^{\alpha}_I \right \rangle_{L^2}^2 \\
& = \sum_{I \in \mathcal{D}} \sum_{\alpha \in \Gamma_d}  C_I(w,\alpha)^2\langle w \rangle^{-1}_I 
\left \langle f, h_I^{w,\alpha} \right \rangle_{L^2}^2 \\
&\ \ + 2\sum_{I \in \mathcal{D}} \sum_{\alpha\in \Gamma_d}  C_I(w,\alpha)D_I(w,\alpha) \langle w \rangle^{-1}_I 
 \left \langle f, h_I^{w,\alpha}  \right \rangle_{L^2} \left \langle f \right \rangle_{E_{\alpha, I}}  \\
& \ \ +
\sum_{I \in \mathcal{D}} \sum_{\alpha \in \Gamma_d} D_I(w,\alpha)^2 \langle w \rangle^{-1}_I 
 \left \langle f  \right \rangle_{E_{\alpha,I}}^2 \\
&=  S_1 + S_2 + S_3.
\end{aligned}
\]
 Since by (\ref{C_Kest}), each $C_I(w,\alpha)^2 \lesssim \left \langle w \right \rangle_I$, 
 we can conclude that each $C_I(w,\alpha)^2\langle w \rangle^{-1}_I  \lesssim 1$. This means
\[ 
S_1 \lesssim \sum_{I \in \mathcal{D}} \sum_{\alpha \in \Gamma_d} \left | \langle f, h^{w,\alpha}_I 
\rangle_{L^2}  \right |^2 = \sum_{I \in \mathcal{D}} \sum_{\alpha \in \Gamma_d}\left | \langle w^{-1} f, h^{w,\alpha}_I  
\rangle_{L^2(w)}  \right |^2  \le \| f \|^2_{L^2(w^{-1})}.
\]
Observe that 
\[ S_2 \lesssim  \left( \sum_{I \in \mathcal{D}} \sum_{\alpha \in \Gamma_d}  C_I(w,\alpha)^2 \langle w \rangle^{-1}_I 
  \left |  \left \langle f, h_I^{w,\alpha}  \right \rangle_{L^2} \right |^2   \right)^{\frac{1}{2}}\left( \sum_{I \in \mathcal{D}} \sum_{\alpha \in \Gamma_d} D_I(w,\alpha)^2 \langle w \rangle^{-1}_I 
\left \langle f \right \rangle_{E_{\alpha,I}}^2  \right)^{\frac{1}{2}}.
\]
The first part of the product is the square root of $S_1$ and the second part
is the square root of $S_3.$ Thus, the proof is reduced to controlling $S_3.$ We use the 
modified Carleson Embedding Theorem. To apply it, we need 
\[ 
\begin{aligned}
 \frac{1}{|E_{\alpha,I}|} \sum_{J \subset I } \sum_{\beta: E_{\beta, J} \subset E_{\alpha, I}} D_J(w,\beta)^2 \langle w \rangle^{-1}_J 
\langle w \rangle^{2}_{E_{\beta,J}}  &= \frac{1}{|E_{\alpha,I}|} \sum_{J \subset I } \sum_{\beta: E_{\beta, J} \subset E_{\alpha, I}} 
\frac{\widehat{w}(J,\beta)^2}{\left\langle w \right\rangle_{E_{\beta, J}}^2} \langle w \rangle^{-1}_J \langle w \rangle^{2}_{E_{\beta,J}} \\
& \lesssim \frac{1}{|E_{\alpha,I}|} \sum_{J \subset I } \sum_{\beta: E_{\beta, J} \subset E_{\alpha, I}} \widehat{w}(J,\beta)^2\langle w \rangle^{-1}_{E_{\beta, J}}\\
& \lesssim   [w]_{A_2} \left \langle w \right \rangle_{E_{\alpha,I}},
\end{aligned}
\]
where the last inequality appears in \cite{C}*{Proposition 4.9, (4.17)}. Then, the modified Carleson Embedding Theorem implies 
\[
\begin{aligned} 
S_3 &= \sum_{I \in \mathcal{D}} \sum_{\alpha \in \Gamma_d} D_I(w,\alpha)^2 \langle w \rangle^{-1}_I 
 \left \langle f w^{\frac{1}{2}} w^{-\frac{1}{2}}  \ \right \rangle_{E_{\alpha,I}}^2
 \lesssim [w]_{A_2} \| f w^{-\frac{1}{2}}  \|^2_{L^2} =  [w]_{A_2} \| f \|^2_{L^2(w^{-1})},
 \end{aligned}
 \]
which proves the lower square function bound. Given (\ref{thm:SquareEst2}) for every $A_2$ weight, 
the upper square function bound follows almost immediately.
Now, for $w$ with $w, w^{-1}$ bounded, the desired inequality is equivalent to 
\[ \langle D_w f , f \rangle_{L^2} \lesssim  [w]^2_{A_2} \langle M_w f,f \rangle_{L^2}, \quad \forall f \in L^2. \]
To prove that, we require the following operator inequality
\[ D_w \le [w]_{A_2} \left( D_{w^{-1}} \right)^{-1}. \]
This is immediate since the trivial inequality $\langle w \rangle_I \le [w]_{A_2} \langle w^{-1} \rangle_I^{-1}$ implies
\[
\langle D_w f ,f \rangle_{L^2} = \sum_{I \in \mathcal{D}}  \langle w \rangle_I  \left | \widehat{f}(I) \right |^2  
 \le  [w]_{A_2} \sum_{I \in \mathcal{D}}  \langle  w^{-1} \rangle_I^{-1}  \left | \widehat{f}(I) \right |^2 
=  [w]_{A_2} \langle  \left( D_{w^{-1}} \right)^{-1} f ,f \rangle_{L^2}.
\]
Combining that estimate with $(\ref{eqn:est2})$ applied to $w^{-1}$ gives:
\[ \langle D_w f ,f \rangle_{L^2} \le [w]_{A_2} \langle  \left( D_{w^{-1}} \right)^{-1} f ,f \rangle_{L^2}
\lesssim  [w]_{A_2}^2 \langle \left(M_{w^{-1}}\right)^{-1} f ,f  \rangle_{L^2} =  [w]_{A_2}^2 \langle M_w f ,f \rangle_{L^2} , \qquad \forall f \in L^2,
\]
which completes the proof. \end{proof}

\subsubsection{Key Estimates Deduced from the Square Function}

As we have shown above, for $w\in A_2$,
\begin{equation*}
\left\Vert Sf \right\Vert_{L^2\left(w\right)}^2\lesssim \left[w\right]%
^2_{A_2}\left\Vert f\right\Vert_{L^2(w)}^2.
\end{equation*}
Applying this inequality to $f=w^{-\frac{1}{2}}\mathbf{1}_{E_{\alpha,I}}$ for $I\in%
\mathcal{D}, \alpha \in \Gamma_d$ and using some trivial estimates yields the following: 
\begin{equation}  \label{SquareFunctionEst}
\sum_{J\subseteq I} \sum_{\beta: E_{\beta, J} \subseteq E_{\alpha, I}} \left\vert \widehat{w^{-\frac{1}{2}}}(J, \beta)\right\vert^2
\left\langle w\right\rangle_J\lesssim \left[w\right]^2_{A_2}\vert
E_{\alpha, I} \vert\quad \forall I\in\mathcal{D}, \alpha \in \Gamma_d.
\end{equation}
A trivial consequence of \eqref{SquareFunctionEst} is the following: 
\begin{equation}  \label{SquareFunctionEst2}
\sum_{J\subseteq I} \sum_{\beta: E_{\beta, J} \subseteq E_{\alpha, I}} \left\vert \widehat{w^{-\frac{1}{2}}}(J, \beta)\right\vert^2
\left\langle w^{\frac{1}{2}}\right\rangle_{E_{\beta,J}}^2\lesssim \left[w\right]
^2_{A_2}\vert E_{\alpha, I} \vert\quad \forall I\in\mathcal{D}, \alpha \in \Gamma_d,
\end{equation}
since $\left\langle w^{\frac{1}{2}}\right\rangle_{E_{\beta,J}}^2 \lesssim \left\langle w^{\frac{1}{2}}\right\rangle_{J}^2\leq \left\langle
w\right\rangle_{J}$. Applying the linear bound of the square function to 
$w^{-1}\mathbf{1}_{E_{\alpha, I}}$, again using trivial estimates, yields
\begin{equation}  \label{SquareFunctionEst3}
\sum_{J\subseteq I} \sum_{\beta: E_{\beta, J} \subseteq E_{\alpha, I}} \left\vert \widehat{w^{-1}}(J, \beta)\right\vert^2
\left\langle w\right\rangle_{E_{\beta, J}}\lesssim \left[w\right]%
^2_{A_2} w^{-1}\left( E_{\alpha, I}\right)\quad \forall I\in\mathcal{D}, \alpha \in \Gamma_d.
\end{equation}
Because of the symmetry of the $A_2$ condition, we additionally have these
estimates with the roles of $w$ and $w^{-1}$ interchanged. These estimates
will all play a fundamental role when applying the modified Carleson Embedding
Theorem.

\section{Linear Bound for Haar Multipliers}
\label{sec:Haar}

We now turn to proving Thereom \ref{thm:haarmultiplierbound}.  
Given a sequence $\sigma=\left\{\sigma_{I,\alpha}\right\}_{I\in\mathcal{D},\alpha\in\Gamma_d}$ we define the Haar multiplier by 
$$
T_\sigma f\equiv\sum_{\alpha\in\Gamma_d} \sum_{I\in\mathcal{D}} \sigma_{I,\alpha} \widehat{f}(I,\alpha) h_I^{\alpha} \qquad \forall f \in L^2.
$$
We must show that 
$$
\left\Vert Q_{T_\sigma,w}^{\left(
\varepsilon _{1},\varepsilon _{2}\right) ,\left( \varepsilon
_{3},\varepsilon _{4}\right) }\right\Vert_{L^2\to L^2}\lesssim \left[ w\right]_{A_2}\left\Vert \sigma\right\Vert_{\infty}
$$
where the operators $Q_{T_\sigma,w}^{\left(
\varepsilon _{1},\varepsilon _{2}\right) ,\left( \varepsilon
_{3},\varepsilon _{4}\right) }$ are defined via \eqref{canoncialpara} and the canonical decomposition of $M_{w^{ \pm \frac{1}{2}}}$ into paraproducts 
is given in \eqref{Multiplication}.

\subsection{Estimating the Easy Terms}

There are four easy terms that arise from \eqref{canoncialpara}. They are
easy because the composition of the paraproducts reduce to classical
paraproduct type operators. The terms are: 
\begin{equation}
\mathsf{P}_{\widehat{w^{\frac{1}{2}}}}^{(1,0)}T_{\sigma}\mathsf{P}_{%
\widehat{w^{-\frac{1}{2}}}}^{(0,1)};  \label{Easy1}
\end{equation}
\begin{equation}  \label{Easy2}
\mathsf{P}_{\widehat{w^{\frac{1}{2}}}}^{(1,0)} T_{\sigma}\mathsf{P}%
_{\langle w^{-\frac{1}{2}}\rangle}^{(0,0)};
\end{equation}
\begin{equation}  \label{Easy3}
\mathsf{P}_{\langle w^{\frac{1}{2}}\rangle}^{(0,0)}T_{\sigma}\mathsf{P}_{%
\widehat{w^{-\frac{1}{2}}}}^{(0,1)};
\end{equation}
\begin{equation}  \label{Easy4}
\mathsf{P}_{\langle w^{\frac{1}{2}}\rangle}^{(0,0)}T_{\sigma}\mathsf{P}%
_{\langle w^{-\frac{1}{2}}\rangle}^{(0,0)}.
\end{equation}
For these terms, we proceed by computing the norm of the operators in
question by using duality. Key to this will be the application of the
modified Carleson Embedding Theorem. 

\subsubsection{Estimating $P^{(1,0)}_{\widehat{w^{\frac{1}{2}}}}
T_{\sigma}P^{(0,1)}_{\widehat{w^{ - \frac{1}{2}}}}$}  

\noindent Fix $\phi,\psi \in L^2$ and observe that 
\begin{eqnarray*}
P^{(1,0)}_{\widehat{w^{\frac{1}{2}}}}T_{\sigma}
P^{(0,1)}_{\widehat{w^{ - \frac{1}{2}}}} \phi & = & 
P^{(1,0)}_{\widehat{w^{\frac{1}{2}}}} \sum_{\alpha\in\Gamma_d}  \sum_{I \in \mathcal{D}}
 \sigma_{I,\alpha} \langle \phi \rangle_{E_{\alpha, I}}  \widehat{w^{- \frac{1}{2}}}(I,\alpha) h_I^{\alpha} \\
& = & \sum_{\alpha\in\Gamma_d} \sum_{I \in \mathcal{D}} \sigma_{I,\alpha} \widehat{w^{\frac{1}{2}}}(I,\alpha)
 \widehat{w^{- \frac{1}{2}}}(I,\alpha) \langle \phi \rangle_{E_{\alpha,I}}  h^{1}_{E_{\alpha,I}}. 
\end{eqnarray*}
Then, we can calculate:
\begin{eqnarray}
\left | \left \langle P^{(1,0)}_{\widehat{w^{\frac{1}{2}}}}
T_{\sigma}P^{(0,1)}_{\widehat{w^{ - \frac{1}{2}}}}\phi,\psi \right \rangle_{L^2} \right |
&=& \left |  \sum_{\alpha\in\Gamma_d}\sum_{I \in \mathcal{D}} \sigma_{I,\alpha} \widehat{w^{\frac{1}{2}}} (I,\alpha) \widehat{w^{- 
\frac{1}{2}}}(I,\alpha) \langle \phi \rangle_{E_{\alpha, I}}  \langle \psi \rangle_{E_{\alpha,I}} \right | \label{P1001}\\
& \le&   \| \sigma \|_{\infty} \left(  \sum_{\alpha\in\Gamma_d} \sum_{I \in \mathcal{D}}  | \widehat{w^{\frac{1}{2}}}(I,\alpha) 
\widehat{w^{- \frac{1}{2}}}(I,\alpha)|  \langle \phi \rangle_{E_{\alpha,I}}^2 \right)^{\frac{1}{2}} \notag\\
 &  & \times  \left(\sum_{\alpha\in\Gamma_d}\sum_{I \in \mathcal{D}} | \widehat{w^{\frac{1}{2}}}(I,\alpha) \widehat{w^{- \frac{1}{2}}}(I,\alpha)|  
\langle \psi \rangle_{E_{\alpha,I}}^2 \right)^{\frac{1}{2}} \notag \\
&\lesssim& \| \sigma \|_{\infty} [ w ]^{\frac{1}{2}}_{A_2} \| \phi \|_{L^2} \| \psi \|_{L^2}\notag,
\end{eqnarray}
where the last inequality follows from the Carleson Embedding Theorem. It applies here, since Cauchy-Schwarz gives:
\begin{eqnarray}
\frac{1}{\left\vert E_{\alpha,I}\right\vert} 
\sum_{J \subset I}\sum_{\beta: E_{\beta,J}\subset E_{\alpha,I}} \left\vert \widehat{w^{\frac{1}{2}}}(I,\beta)  \widehat{w^{- \frac{1}{2}}}(I,\beta)\right\vert & 
\le &   
\frac{1}{\left\vert E_{\alpha,I}\right\vert} \left \| w^{\frac{1}{2}} \textbf{1}_{E_{\alpha,I}} \right\|_{L^2} \left \| w^{-\frac{1}{2}} \textbf{1}_{E_{\alpha,I}} \right \|_{L^2} \notag\\
& = &  \left( \langle w \rangle_{E_{\alpha,I}} \langle w^{-1} \rangle_{E_{\alpha,I}} \right)^{\frac{1}{2}} \notag \\
& \lesssim&  [w]^{\frac{1}{2}}_{A_2}. \label{eqn:estimate1}
\end{eqnarray}
Taking the supremum over all $\phi, \psi \in L^2$ and using duality gives
\[
\left \| P^{(1,0)}_{\widehat{w^{\frac{1}{2}}}}
T_{\sigma}P^{(0,1)}_{\widehat{w^{ - \frac{1}{2}}}} \right \|_{L^2 \rightarrow L^2} 
\lesssim  \| \sigma \|_{\infty} [w]^{\frac{1}{2}}_{A_2}  \le  \| \sigma \|_{\infty} [w]_{A_2}.
\]
\subsubsection{Estimating $P^{(1,0)}_{\widehat{w^{\frac{1}{2}}}}
 T_{\sigma} P^{(0,0)}_{\langle w^{- \frac{1}{2}} \rangle}$ 
and $P^{(0,0)}_{\langle w^{ \frac{1}{2}} \rangle }T_{\sigma}
P^{(0,1)}_{\widehat{w^{- \frac{1}{2}}}}$} 

\noindent As these two operators are symmetric, very similar
arguments can be used to control both of them. 
Thus, we only provide details for the first operator. Observe that 
\[
P^{(1,0)}_{\widehat{w^{\frac{1}{2}}}}T_{\sigma}
P^{(0,0)}_{\langle w^{- \frac{1}{2}} \rangle} \phi = 
P^{(1,0)}_{\widehat{w^{\frac{1}{2}}}} \sum_{\alpha\in\Gamma_d} \sum_{I \in 
\mathcal{D}} \sigma_{I,\alpha} \langle w^{- \frac{1}{2}} \rangle_{E_{\alpha,I}}  \widehat{\phi}(I,\alpha)  h_I^{\alpha} 
= \sum_{\alpha\in\Gamma_d} \sum_{I \in \mathcal{D}} \sigma_{I,\alpha} \langle w^{- \frac{1}{2}} \rangle_{E_{\alpha,I}} \widehat{w^{\frac{1}{2}}}(I,\alpha) \widehat{\phi}(I,\alpha) \,h^{1}_{E_{\alpha,I}}. 
\]
Fixing $\phi,\psi \in L^2$, we can calculate
\begin{eqnarray}
\left\vert \left \langle P^{(1,0)}_{\widehat{w^{\frac{1}{2}}}}T_{\sigma}P^{(0,0)}_{\langle
 w^{- \frac{1}{2}} \rangle} \phi, \psi \right \rangle_{L^2} \right \vert
& = & \left \vert  \sum_{\alpha\in\Gamma_d}\sum_{I \in \mathcal{D}} \sigma_{I,\alpha} \langle 
w^{- \frac{1}{2}} \rangle_{E_{\alpha,I}} \widehat{w^{\frac{1}{2}}}(I,\alpha) \widehat{\phi}(I,\alpha) \langle \psi \rangle_{E_{\alpha, I}} \right \vert \label{P1000orP0001}\\
& \le & \| \sigma \|_{\infty} \| \phi \|_{L^2} \left( \sum_{\alpha\in\Gamma_d}\sum_{I \in \mathcal{D}} 
\left\vert \widehat{w^{\frac{1}{2}}}(I,\alpha)\right\vert^2 \langle w^{- \frac{1}{2}} 
\rangle_{E_{\alpha,I}}^2 \langle \psi \rangle_{E_{\alpha,I}}^2 \right )^{\frac{1}{2}} \notag \\
& \lesssim & \| \sigma \|_{\infty} [w]_{A_2} \|\phi \|_{L^2} \| \psi \|_{L^2},\notag
\end{eqnarray}
where the last inequality follows from the Carleson Embedding
Theorem and estimate \eqref{SquareFunctionEst2}.  Again, taking the 
supremum over all $\phi, \psi \in L^2$ and using duality gives
\[
\left \| P^{(1,0)}_{\widehat{w^{\frac{1}{2}}}}
 T_{\sigma} P^{(0,0)}_{\langle w^{- \frac{1}{2}} \rangle} \right \|_{L^2 \rightarrow L^2} 
\lesssim  \| \sigma \|_{\infty} [w]_{A_2}.
\]

\subsubsection{Estimating $P^{(0,0)}_{\langle w^{ \frac{1}{2}}
 \rangle }T_{\sigma}  P^{(0,0)}_{\langle w^{- \frac{1}{2}} \rangle }$} 

\noindent Fixing $\phi, \psi\in L^2$, observe that  
\begin{eqnarray*}
P^{(0,0)}_{\langle w^{ \frac{1}{2}} \rangle }T_{\sigma}  
P^{(0,0)}_{\langle w^{- \frac{1}{2}} \rangle} \phi 
& = &  P^{(0,0)}_{\langle w^{ \frac{1}{2}} \rangle }\sum_{\alpha\in\Gamma_d} \sum_{I \in 
\mathcal{D}} \sigma_{I,\alpha} \left\langle w^{- \frac{1}{2}} \right\rangle_{E_{\alpha,I}} \widehat{\phi}(I,\alpha) h_I^{\alpha} \\
 & = & \sum_{\alpha\in\Gamma_d}\sum_{I \in \mathcal{D}} \sigma_{I,\alpha} \left\langle w^{ \frac{1}{2}} 
\right\rangle_{E_{\alpha,I}} \left\langle w^{- \frac{1}{2}} \right\rangle_{E_{\alpha,I}} \widehat{\phi}(I,\alpha) h_I^{\alpha}.
\end{eqnarray*}
This means we can calculate
\begin{eqnarray}
\left | \left \langle P^{(0,0)}_{\langle w^{ \frac{1}{2}} \rangle }
T_{\sigma}  P^{(0,0)}_{\langle w^{- \frac{1}{2}} \rangle} \phi , \psi \right \rangle_{L^2} \right | 
& = &  \left | \sum_{\alpha\in\Gamma_d}\sum_{I \in \mathcal{D}} \sigma_{I,\alpha} \left\langle w^{ 
\frac{1}{2}} \right\rangle_{E_{\alpha,I}} \left\langle w^{- \frac{1}{2}} \right\rangle_{E_{\alpha,I}} \widehat{\phi}(I,\alpha) \widehat{\psi}(I,\alpha) \right | \label{P0000}\\
& \lesssim &  \| \sigma \|_{\infty} \sup_{I \in \mathcal{D}} \left( \left\langle 
w^{ \frac{1}{2}} \right\rangle_I \left\langle w^{- \frac{1}{2}} \right\rangle_I \right) \| \phi \|_{L^2} \| \psi \|_{L^2} \notag\\
& \le & \| \sigma \|_{\infty} [ w ]^{\frac{1}{2}}_{A_2} \| \phi \|_{L^2} \| \psi \|_{L^2}. \notag
\end{eqnarray}
Taking the supremum over all $\phi,\psi \in L^2$ and using duality gives the 
desired linear norm bound. This concludes the proof for the easy terms.

\subsection{Estimating the Hard Terms} 

There are five remaining terms to be controlled. These include the four
difficult terms:
\begin{equation}
\mathsf{P}_{\widehat{w^{\frac{1}{2}}}}^{(0,1)}T_{\sigma}\mathsf{P}_{%
\widehat{w^{-\frac{1}{2}}}}^{(0,1)} ;
\label{Difficult1} \end{equation}
\begin{equation}\mathsf{P}_{\widehat{w^{\frac{1}{2}}}}^{(0,1)}T_{\sigma}\mathsf{P}_{\langle
w^{-\frac{1}{2}}\rangle}^{(0,0)};
\label{Difficult2} \end{equation}
\begin{equation}\mathsf{P}_{\widehat{w^{\frac{1}{2}}}}^{(1,0)}T_{\sigma}\mathsf{P}_{%
\widehat{w^{-\frac{1}{2}}}}^{(1,0)} ;
\label{Difficult3} \end{equation}
\begin{equation}\mathsf{P}_{\langle w^{\frac{1}{2}}\rangle}^{(0,0)}T_{\sigma}\mathsf{P}_{%
\widehat{w^{-\frac{1}{2}}}}^{(1,0)}.
\label{Difficult4}
\end{equation}

To estimate terms \eqref{Difficult1} and \eqref{Difficult3} we will rely on
disbalanced Haar functions adapted to the weights $w$ and $w^{-1}$. For these
terms, we also compute the norms using duality and frequent application of the
modified Carleson Embedding Theorem. The proof of the estimates for these
terms is carried out in subsection \ref{subsec:Difficult13}. Terms %
\eqref{Difficult2} and \eqref{Difficult4} will be handled via a similar
method; their analysis appears in subsection \ref{subsec:Difficult24}.

The remaining term is the one for which $T_{\sigma}$ can
not be absorbed into one of the paraproducts. Namely, we need to control the
following expression: 
\begin{equation}
\mathsf{P}_{\widehat{w^{\frac{1}{2}}}}^{(0,1)} T_{\sigma} \mathsf{P}_{%
\widehat{w^{-\frac{1}{2}}}}^{(1,0)}.  \label{VeryDifficult}
\end{equation}%
To handle this term, we must rely on Therem \ref{thm:haarbd} and the 
computed linear bounds for the other eight paraproduct compositions.
This leaves the open the question of whether there is an independent proof
of the linear bound for \eqref{VeryDifficult}. This is discussed further 
in subsection \ref{subsec:VeryDifficult}.

\subsubsection{Estimating $P^{(0,1)}_{\widehat{
w^{ \frac{1}{2}}}}T_{\sigma}P^{(0,1)}_{\widehat{w^{ - \frac{1}{2}}}}$ and
$P^{(1,0)}_{\widehat{w^{\frac{1}{2}}}}T_{\sigma}P^{(1,0)}_{
\widehat{w^{- \frac{1}{2}}}}$}  \label{subsec:Difficult13}
Similar arguments handle both terms and so, we
restrict attention to the first one. Fix $\phi,\psi \in L^2.$ Observe that
 basic manipulations and the product formula \eqref{product formula} 
 for Haar coefficients give
\[
\begin{aligned}
 \left \langle P^{(0,1)}_{\widehat{w^{ \frac{1}{2}}}}T_{\sigma}
 P^{(0,1)}_{\widehat{w^{ - \frac{1}{2}}}}\phi, \psi \right \rangle_{L^2}
&= \sum_{\alpha\in\Gamma_d} \sum_{I\in \mathcal{D}} \sigma_{I,\alpha}   \langle \phi \rangle_{E_{\alpha,I}} \widehat{w^{-\frac{1}{2}}}(I,\alpha) \sum_{\beta\in\Gamma_d}\sum_{J\in\mathcal{D}} \widehat{w^{ \frac{1}{2}}}(J,\beta) \widehat{\psi}(J,\beta) \left \langle h_I^{\alpha}, h^1_{E_{\beta,J}}  \right \rangle_{L^2} \\
& = \sum_{\alpha\in\Gamma_d}\sum_{I \in \mathcal{D}} \sigma_{I,\alpha}   \langle \phi \rangle_{E_{\alpha,I}}  \widehat{w^{-\frac{1}{2}}}(I,\alpha)  \\
& \ \ \ \ \ \times \left( \widehat{\psi w^{\frac{1}{2}}}(I,\alpha) 
- \widehat{\psi}(I,\alpha) \left \langle w^{\frac{1}{2}} \right \rangle_{E_{\alpha,I}} - \widehat{w
^{\frac{1}{2}}}(I,\alpha) \langle \psi \rangle_{E_{\alpha,I}} \right) \\
& \equiv T_1 + T_2 + T_3.
\end{aligned}
\]
We will show that each $|T_j | \lesssim \| \sigma \|_{\infty} [w]_{A_2}
\| \phi \|_{L^2} \| \psi \|_{L^2}.$ 
The bounds for $T_2$ and $T_3$ follow easily. First, observe that
\[
\begin{aligned}
| T_2|  &\le \| \sigma \|_{\infty}  \sum_{\alpha\in\Gamma_d}\sum_{I \in \mathcal{D}}  \left \vert\langle \phi \rangle_{E_{\alpha,I}} \widehat{w^{-\frac{1}{2}}}
(I,\alpha) \widehat{\psi}(I,\alpha) \left \langle w^{\frac{1}{2}} \right \rangle_{E_{\alpha,I}} \right \vert\\
&\le \| \sigma \|_{\infty} \| \psi \|_{L^2} \left( \sum_{\alpha\in\Gamma_d}\sum_{I \in \mathcal{D}} \left\vert \widehat{w^{-\frac{1}{2}}}
(I,\alpha)\right\vert^2  \left \langle w^{\frac{1}{2}} \right \rangle_{E_{\alpha,I}}^2  \langle \phi \rangle_{E_{\alpha,I}}^2 \right)^{\frac{1}{2}} \\
&\lesssim  \| \sigma \|_{\infty} \left[ w \right] _{A_2}  \| \psi \|_{L^2} \|\phi \|_{L^2}.
\end{aligned}
\]
The last inequality follows via the Carleson Embedding Theorem and 
the square function estimate \eqref{SquareFunctionEst2}. For $T_3$, 
the computations are similarly straightforward:
\[
\begin{aligned}
|T_3| & \le \| \sigma \|_{\infty}  \sum_{\alpha\in\Gamma_d}\sum_{I \in \mathcal{D}} \left |\widehat{w^{-\frac{1}{2}}}(I,\alpha) 
\widehat{w^{\frac{1}{2}}}(I,\alpha) \langle \phi \rangle_{E_{\alpha,I}}\langle \psi \rangle_{E_{\alpha,I}}  \right | \\
& \le \| \sigma \|_{\infty} \left(  \sum_{\alpha\in\Gamma_d}\sum_{I \in \mathcal{D}} | \widehat{w^{-\frac{1}{2}}}(I,\alpha) 
\widehat{w^{\frac{1}{2}}}(I,\alpha) |  \langle \phi \rangle_{E_{\alpha,I}}^2 \right)^{\frac{1}{2}} \left(  
 \sum_{\alpha\in\Gamma_d} \sum_{I \in \mathcal{D}} | \widehat{w^{-\frac{1}{2}}}(I,\alpha)  \widehat{w^{\frac{1}{2}}}(I,\alpha) |  
\langle \psi \rangle_{E_{\alpha,I}}^2 \right)^{\frac{1}{2}}  \\
& \lesssim \| \sigma \|_{\infty} [ w]^{\frac{1}{2}}_{A_2} \| \phi\|_{L^2} \|\psi \|_{L^2}.
\end{aligned}
\]
Here, the last inequality follows from two applications of the Carleson 
Embedding Theorem using the estimate given in \eqref{eqn:estimate1}.
Estimating $T_1$ requires the use of disbalanced Haar functions. We
expand the Haar functions in the sum using two disbalanced systems, one associated 
to $w$ and one associated to $w^{-1}$, as follows:
\begin{eqnarray}\label{P0101orP1010}
T_1 &= & \sum_{\alpha\in\Gamma_d}\sum_{I \in \mathcal{D}} \sigma_{I,\alpha} \langle \phi \rangle_{E_{\alpha,I}}  \widehat{w^{-\frac{1}{2}}}(I,\alpha) \widehat{\psi w^{\frac{1}{2}}}(I,\alpha) \\
& = &   \sum_{\alpha\in\Gamma_d} \sum_{I \in \mathcal{D}} \sigma_{I,\alpha} \langle \phi \rangle_{E_{\alpha,I}} \left \langle w^{-\frac{1}{2}},
 C_I(w^{-1},\alpha)h^{w^{-1},\alpha}_I +D_I(w^{-1},\alpha)h^{1}_{E_{\alpha,I}} \right \rangle_{L^2} \notag \\ 
 && \ \ \ \ \ \ \times \ \left \langle \psi
 w^{\frac{1}{2}},  C_I(w,\alpha)h^{w,\alpha}_I +D_I(w,\alpha)h^{1}_{E_{\alpha,I}} \right \rangle_{L^2}\notag \\
& = & \sum_{\alpha\in\Gamma_d} \sum_{I \in \mathcal{D}} \sigma_{I,\alpha} \langle \phi \rangle_{E_{\alpha,I}} C_I(w^{-1},\alpha)C_I(w,\alpha)\left \langle w^{-\frac{1}{2}},
 h^{w^{-1},\alpha}_I \right \rangle_{L^2}\left \langle \psi
 w^{\frac{1}{2}}, h^{w,\alpha}_I \right \rangle_{L^2} \notag \\
& & \ \ + \sum_{\alpha\in\Gamma_d} \sum_{I \in \mathcal{D}} \sigma_{I,\alpha} \langle \phi \rangle_{E_{\alpha,I}} C_I(w^{-1},\alpha)D_I(w,\alpha)\left \langle w^{-\frac{1}{2}},
 h^{w^{-1},\alpha}_I \right \rangle_{L^2} \left \langle \psi w^{\frac{1}{2}} \right \rangle_{E_{\alpha,I}} \notag \\
& & \ \ + \sum_{\alpha\in\Gamma_d}\sum_{I \in \mathcal{D}} \sigma_{I,\alpha} \langle \phi \rangle_{E_{\alpha,I}} D_I(w^{-1},\alpha)C_I(w,\alpha)\left \langle w^{-\frac{1}{2}}
\right \rangle_{E_{\alpha,I}}  \left \langle \psi w^{\frac{1}{2}}, h^{w,\alpha}_I  \right \rangle_{L^2} \notag \\
& & \ \ + \sum_{\alpha\in\Gamma_d}\sum_{I \in \mathcal{D}} \sigma_{I,\alpha} \langle \phi \rangle_{E_{\alpha,I}} D_I(w^{-1},\alpha)D_I(w,\alpha)
\left \langle w^{-\frac{1}{2}} \right \rangle_{E_{\alpha,I}} \left \langle \psi w^{\frac{1}{2}} \right \rangle_{E_{\alpha,I}} \notag \\ 
&\equiv & S_1 + S_2 + S_3 + S_4 \notag.
\end{eqnarray}
Now, we show each $|S_j | \lesssim \| \sigma \|_{\infty} [w]_{A_2}
\| \phi \|_{L^2} \| \psi \|_{L^2},$ which gives the bound for $T_1.$ Observe
that by (\ref{C_Kest}),
\[
\begin{aligned}
|S_1 | & \le \| \sigma \|_{\infty}  \sum_{\alpha\in\Gamma_d} \sum_{I \in \mathcal{D}} \left\vert 
\langle \phi \rangle_{E_{\alpha,I}} C_I(w^{-1},\alpha)C_I(w,\alpha)\left \langle w^{-\frac{1}{2}},
h^{w^{-1},\alpha}_I \right \rangle_{L^2}\left \langle \psi w^{- \frac{1}{2}}, h^{w,\alpha}_I \right \rangle_{L^2(w)} \right | \\
&\lesssim  \| \sigma\|_{\infty} \left\| \psi w^{-\frac{1}{2}}\right \|_{L^2(w)} \left( \sum_{\alpha\in\Gamma_d}\sum_{I \in \mathcal{D}}  \langle w^{-1} \rangle_{I} \langle w \rangle_{I} \left \langle w^{-\frac{1}{2}},
h^{w^{-1},\alpha}_I \right \rangle_{L^2}^2 \langle \phi \rangle_{E_{\alpha,I}}^2 \right)^{\frac{1}{2}} \\
&\lesssim \| \sigma\|_{\infty} [w]^{\frac{1}{2}}_{A_2} \| \psi \|_{L^2} \left( \sum_{\alpha\in\Gamma_d}\sum_{I \in \mathcal{D}} \left \langle 
w^{-\frac{1}{2}}, h^{w^{-1},\alpha}_I \right \rangle_{L^2}^2 \langle \phi \rangle_{E_{\alpha,I}}^2 \right)^{\frac{1}{2}}  \\
&\lesssim \| \sigma\|_{\infty} [w]^{\frac{1}{2}}_{A_2} \| \psi \|_{L^2}  \| \phi\|_{L^2},
\end{aligned}
\]
where the last inequality followed via the Carleson Embedding Theorem using the estimate
\[
\sum_{J \subset I} \sum_{\beta: E_{\beta, J}\subset E_{\alpha,I}} \left \langle w^{-\frac{1}{2}}, h^{w^{-1},\beta}_J \right \rangle^2_{L^2} 
= \sum_{J \subset I} \sum_{\beta: E_{\beta, J}\subset E_{\alpha,I}}  \left \langle w^{\frac{1}{2}}, h^{w^{-1},\beta}_J \right \rangle^2_{L^2(w^{-1})}
\le \left\| w^{\frac{1}{2}} \textbf{1}_{E_{\alpha,I}} \right\|^2_{L^2(w^{-1})} = \left\vert {E_{\alpha,I}}\right\vert. 
\]
The calculation for $S_2$ is also straightforward:
\[
\begin{aligned}
|S_2| &\le \| \sigma \|_{\infty}  \sum_{\alpha\in\Gamma_d} \sum_{I \in \mathcal{D}} \left\vert \left \langle \phi \right\rangle_{E_{\alpha,I}} C_I(w^{-1},\alpha)D_I(w,\alpha)\left \langle w^{-\frac{1}{2}},
 h^{w^{-1},\alpha}_I \right \rangle_{L^2} \left \langle \psi w^{\frac{1}{2}} \right \rangle_{E_{\alpha,I}}  \right \vert\\
 & \lesssim  \| \sigma \|_{\infty} \left( \sum_{\alpha\in\Gamma_d} \sum_{I \in \mathcal{D}}\left \langle w^{-\frac{1}{2}},
 h^{w^{-1},\alpha}_I \right \rangle_{L^2}^2 \langle \phi \rangle_{E_{\alpha,I}}^2 \right)^{\frac{1}{2}}   
 \left( \sum_{\alpha\in\Gamma_d}\sum_{I \in \mathcal{D}} \langle w^{-1} \rangle_{E_{\alpha,I}} \frac{\left\vert \widehat{w}(I,\alpha)\right\vert^2}{\langle w \rangle_{E_{\alpha, I}}^2}\left \langle \psi w^{\frac{1}{2}} \right \rangle_{E_{\alpha,I}}^2  \right)^{\frac{1}{2}} \\
&\lesssim \| \sigma \|_{\infty} [ w]_{A_2} \| \phi \|_{L^2} \| \psi \|_{L^2},
\end{aligned}
\]
where we use the Carleson Embedding Theorem twice. 
The application for the $\phi$ term follows as in the estimate 
for $S_1$, while the application for $\psi $ follows from the square function estimate \eqref{SquareFunctionEst3}. 
Similarly, for $S_3$, we can calculate
\[
\begin{aligned}
\left\vert S_3 \right\vert & \le \| \sigma \|_{\infty} \sum_{\alpha\in\Gamma_d} \sum_{I \in \mathcal{D}} \left | \langle \phi 
\rangle_{E_{\alpha,I}} D_I(w^{-1},\alpha)C_I(w,\alpha)\left \langle w^{-\frac{1}{2}}
\right \rangle_{E_{\alpha,I}}  \left \langle \psi w^{\frac{1}{2}}, h^{w,\alpha}_I  \right \rangle_{L^2} \right |\\
& \lesssim \| \sigma \|_{\infty}  \left( \sum_{\alpha\in\Gamma_d} \sum_{I \in \mathcal{D}} \frac{ \left\vert \widehat{w^{-1}}
(I,\alpha)\right\vert ^2}{\langle w^{-1} \rangle^2_{E_{\alpha,I}}}  \langle \phi \rangle_{E_{\alpha,I}}^2 \right)^{\frac{1}{2}} 
\left( \sum_{\alpha\in\Gamma_d} \sum_{I \in \mathcal{D}}  \langle w \rangle_{I} \left \langle w^{-\frac{1}{2}}
\right \rangle_I^2 \left \langle \psi w^{-\frac{1}{2}}, h^{w,\alpha}_I  \right \rangle_{L^2(w)}^2 \right)^{\frac{1}{2}} \\
& \lesssim  \| \sigma \|_{\infty} [ w]^{\frac{1}{2}}_{A_2}  \left\| \psi w^{-\frac{1}{2}} \right\|_{L^2(w)}  \left(  \sum_{\alpha\in\Gamma_d}\sum_{I \in \mathcal{D}} \frac{ \left\vert \widehat{w^{-1}}
(I,\alpha)\right\vert^2}{\langle w^{-1} \rangle^2_{E_{\alpha,I}}}  \langle \phi \rangle_{E_{\alpha,I}}^2 \right)^{\frac{1}{2}} \\
& \lesssim \| \sigma \|_{\infty} [ w]_{A_2} \| \phi \|_{L^2} \|\psi \|_{L^2},
\end{aligned}
\]
where the last inequality follows from the Carleson Embedding Theorem. 
It applies here since:
\begin{equation}
\label{FKP_Wilson}
\sum_{J\subset I} \sum_{\beta: E_{\beta,J}\subset E_{\alpha,I}} \frac{\left\vert \widehat{w^{-1}}(J,\beta)\right\vert^2}{\left\langle w^{-1}\right\rangle_{E_{\beta,J}}^2}  \lesssim  \left[w\right]_{A_2} \left\vert E_{\alpha, I}\right\vert \quad\forall \alpha\in\Gamma_d\quad\forall I\in\mathcal{D}.
\end{equation}
To see that \eqref{FKP_Wilson} holds, it is then a simple application of Cauchy-Schwarz and the following estimates:
\begin{eqnarray}
\sum_{J\subset I} \sum_{\beta: E_{\beta,J}\subset E_{\alpha,I}} \frac{\left\vert \widehat{w^{-1}}(J,\beta)\right\vert^2}{\left\langle w^{-1}\right\rangle_{E_{\beta,J}}} & \lesssim & \left[w\right]_{A_2} w^{-1}(E_{\alpha, I})\quad\forall \alpha\in\Gamma_d\quad\forall I\in\mathcal{D}; \label{power1}\\
\sum_{J\subset I} \sum_{\beta: E_{\beta,J}\subset E_{\alpha,I}} \frac{\left\vert  \widehat{w^{-1}}(J,\beta)\right\vert^2}{\left\langle w^{-1}\right\rangle_{E_{\beta,J}}^3} & \lesssim & w(E_{\alpha,I})\quad\forall \alpha\in\Gamma_d\quad\forall I\in\mathcal{D}\label{power3}.
\end{eqnarray}
The proofs of \eqref{power1} and \eqref{power3} can be found in \cite{C}*{Proposition 4.9, Equation (4.17)} and \cite{C}*{Proposition 4.7, Equation (4.11)} respectively.  

Lastly, the estimate for $S_4$ is computed as follows:
\[
\begin{aligned}
\left\vert S_4 \right\vert  & \le \| \sigma \|_{\infty}  \sum_{\alpha\in\Gamma_d} \sum_{I \in \mathcal{D}}\left \vert
\langle \phi \rangle_{E_{\alpha,I}} D_I(w^{-1},\alpha)D_I(w,\alpha)\left \langle w^{-\frac{1}{2}}
\right \rangle_{E_{\alpha,I}} \left \langle \psi w^{\frac{1}{2}} \right \rangle_{E_{\alpha,I}} \right \vert \\
& \le \| \sigma \|_{\infty} \left( \sum_{\alpha\in\Gamma_d} \sum_{I \in \mathcal{D}} 
\left\vert \widehat{w}(I,\alpha) \widehat{w^{-1}}(I,\alpha)\right\vert \langle \phi \rangle_{E_{\alpha,I}}^2 \right)^{\frac{1}{2}}
\left( \sum_{\alpha\in\Gamma_d}\sum_{I \in \mathcal{D}} \frac{ \left\vert \widehat{w}(I,\alpha) 
\widehat{w^{-1}}(I,\alpha)\right\vert}{ \langle w \rangle_{E_{\alpha,I}}^2 \langle w^{-1} \rangle_{E_{\alpha,I}} }\left \langle \psi w^{\frac{1}{2}} \right \rangle_{E_{\alpha,I}}^2 \right)^{\frac{1}{2}} \\
&= \| \sigma \|_{\infty} [w]_{A_2} \| \phi \|_{L_2} \|\psi \|_{L^2}, 
\end{aligned}
\]
where the Carleson Embedding Theorem is used twice. The application for 
the $\phi$ term uses 
\begin{equation}
 \label{eqn:WitDyadicSum}
\sum_{J\subset I} \sum_{\beta: E_{\beta, J}\subset E_{\alpha,I}} \left | \widehat{w}(J,\beta)\, \widehat{w^{-1}}(J,\beta) \right | \lesssim \left[w\right]_{A_2}\left\vert E_{\alpha, I}\right\vert\quad\forall\alpha\in\Gamma_d\quad\forall  I\in\mathcal{D}.
\end{equation}
As stated, the proof of this is found in \cite{C}*{Equation (6.3)}.  The one-dimensional version is established in \cite{wit00}*{Lemma 4.7}.  The application for the $\psi$ term uses 
\begin{equation}
\sum_{J\subset I} \sum_{\beta: E_{\beta, J}\subset E_{\alpha,I}} \frac{ \left | \widehat{w}(J,\beta)\, \widehat{w^{-1}}(J,\beta) \right|}{\left\langle w^{-1}\right\rangle_{E_{\beta,J}}}\lesssim \left[w\right]_{A_2} w\left(E_{\alpha, I}\right) \quad\forall\alpha\in\Gamma_d\quad\forall  I\in\mathcal{D}.
\end{equation}
As stated, the proof of this is found in \cite{C}*{Equation (6.4)}.

This concludes the proof of the estimates for $T_1, T_2, T_3.$ By taking the supremum over $\phi,\psi \in L^2$ and using  duality, we conclude that
\[ 
\left \Vert P^{(0,1)}_{\widehat{
w^{ \frac{1}{2}}}}T_{\sigma}P^{(0,1)}_{\widehat{w^{ - \frac{1}{2}}}} 
\right \Vert_{L^2\to L^2} \lesssim  \| \sigma \|_{\infty} [w]_{A_2}.
\]
\subsubsection{Estimating $P^{(0,1)}_{\widehat{w^{ \frac{1}{2}}}}
T_{\sigma} P^{(0,0)}_{\langle w^{- \frac{1}{2}} \rangle}$ 
and $P^{(0,0)}_{\langle w^{ \frac{1}{2}} \rangle }T_{\sigma}
P^{(1,0)}_{\widehat{w^{ - \frac{1}{2}}}}$ } \label{subsec:Difficult24}

We only discuss the first operator, as the estimates for the 
second one follow via similar arguments. Fix $\phi,\psi \in L^2$. We first
simplify using basic manipulations and the product formula for Haar coefficients, \eqref{product formula}, as follows:
\[
\begin{aligned}
\left \langle P^{(0,1)}_{\widehat{w^{ \frac{1}{2}}}}T_{\sigma} P^{(0,0)}_{
\langle w^{- \frac{1}{2}} \rangle}\phi,\psi \right \rangle_{L^2}
& =  \sum_{\alpha\in\Gamma_d}\sum_{I\in \mathcal{D}} \sigma_{I,\alpha} \left \langle w^{-\frac{1}{2}} \right \rangle_{E_{\alpha,I}} \widehat{\phi}(I,\alpha) \sum_{\beta\in\Gamma_d}\sum_{J\in\mathcal{D}} \widehat{w^{\frac{1}{2}}}(J,\beta) \widehat{\psi}(J,\beta) \left \langle h_I^{\alpha}, h^1_{E_{\beta,J}} \right \rangle_{L^2}  \\
& = \sum_{\alpha\in\Gamma_d}\sum_{I \in \mathcal{D}}  \sigma_{I,\alpha} \left \langle w^{-\frac{1}{2}} \right \rangle_{E_{\alpha,I}} \widehat{\phi}(I,\alpha) \\
 & \ \ \ \  \times \left( \widehat{\psi w^{\frac{1}{2}}}(I,\alpha) - \widehat{\psi}(I,\alpha)\left \langle w^{\frac{1}{2}} \right \rangle_{E_{\alpha,I}}
- \widehat{w^{\frac{1}{2}}}(I, \alpha)\langle \psi \rangle_{E_{\alpha,I}}  \right)\\
& \equiv T_1 +T_2 +T_3.
\end{aligned}
\]
As in the previous case, we show that each $\left\vert T_j\right\vert \lesssim \| 
\sigma \|_{\infty} [w]_{A_2} \|\phi \|_{L^2} \|\psi \|_{L^2}.$
The estimates for $T_2$ and $T_3$ follow easily. Observe that
\[
\begin{aligned}
 |T_2| &\le \| \sigma \|_{\infty}  \sum_{\alpha\in\Gamma_d}\sum_{I \in \mathcal{D}} 
\left | \left \langle w^{-\frac{1}{2}} \right \rangle_{E_{\alpha,I}} \left \langle w^{\frac{1}{2}} \right \rangle_{E_{\alpha,I}}
 \widehat{\phi}(I,\alpha)  \widehat{\psi}(I,\alpha)\right | \\
 &\lesssim  \| \sigma \|_{\infty}  [ w]^{\frac{1}{2}}_{A_2}  \sum_{\alpha\in\Gamma_d}\sum_{I \in \mathcal{D}}   \left |
 \widehat{\phi}(I,\alpha) \widehat{\psi}(I,\alpha)\right | \\
 & \le \|\sigma \|_{\infty} [ w]^{\frac{1}{2}}_{A_2} \|\phi  \|_{L^2} \| \psi \|_{L^2}, 
\end{aligned}
\]
and similarly,
\[
\begin{aligned}
| T_3| & \le \|\sigma\|_{\infty}  \sum_{\alpha\in\Gamma_d}\sum_{I \in \mathcal{D}} \left|
\left \langle w^{-\frac{1}{2}} \right \rangle_{E_{\alpha,I}} \langle \psi \rangle_{E_{\alpha,I}} \widehat{\phi}(I,\alpha)
 \widehat{w^{\frac{1}{2}}}(I,\alpha) \right | \\
& \le \| \sigma \|_{\infty} \| \phi\|_{L^2} \left(\sum_{\alpha\in\Gamma_d}
\sum_{I \in \mathcal{D}} 
\left \langle w^{-\frac{1}{2}} \right \rangle^2_{E_{\alpha,I}} 
 \left\vert \widehat{w^{\frac{1}{2}}}(I,\alpha)\right\vert^2 \langle \psi \rangle_{E_{\alpha,I}}^2 \right)^{\frac{1}{2}} \\
& \lesssim \| \sigma \|_{\infty} [w]_{A_2} \| \phi\|_{L^2} \|\psi \|_{L^2},
\end{aligned}
\]
where the last inequality followed via an application of the Carleson
Embedding Theorem using \eqref{SquareFunctionEst2}. 

To estimate $T_1$, we rewrite the term $\left\langle \psi w^{\frac{1}{2}}, h^{\alpha}_I \right\rangle_{L^2}$ 
using disbalanced Haar functions adapted to $w$ as follows:
\begin{eqnarray}
T_1 &= & \sum_{\alpha\in\Gamma_d} \sum_{I \in \mathcal{D}}  \sigma_{I,\alpha} \left \langle w^{-\frac{1}{2}} \right \rangle_{E_{\alpha,I}} \widehat{\phi}(I,\alpha) \widehat{\psi w^{\frac{1}{2}}}(I,\alpha)\label{T_1 for 0100 or 0001}\\
&  = & \sum_{\alpha\in\Gamma_d} \sum_{I \in \mathcal{D}}  \sigma_{I,\alpha} \left \langle 
w^{-\frac{1}{2}} \right \rangle_{E_{\alpha,I}} \widehat{\phi}(I,\alpha)\left \langle 
\psi w^{\frac{1}{2}}, C_I(w,\alpha) h^{w,\alpha}_I + D_I(w,\alpha)h^{1}_{E_{\alpha,I}} \right \rangle_{L^2}  \notag\\
& = & \sum_{\alpha\in\Gamma_d} \sum_{I \in \mathcal{D}}  \sigma_{I,\alpha} C_I(w,\alpha) \left \langle 
w^{-\frac{1}{2}} \right \rangle_{E_{\alpha,I}} \widehat{\phi}(I,\alpha)\left \langle 
\psi w^{\frac{1}{2}}, h^{w,\alpha}_I \right \rangle_{L^2}\notag\\
& & +  \sum_{\alpha\in\Gamma_d} \sum_{I \in
\mathcal{D}}  \sigma_{I,\alpha} D_I(w,\alpha) \left \langle w^{-\frac{1}{2}} \right 
\rangle_{E_{\alpha,I}} \widehat{\phi}(I,\alpha) \left \langle \psi w^{\frac{1}{2}}\right \rangle_{E_{\alpha,I}} \notag\\
& = & S_1 + S_2.\notag
\end{eqnarray}
Now, we show each $|S_j | \lesssim \| \sigma \|_{\infty} [w]_{A_2}
\| \phi\|_{L^2} \| \psi \|_{L^2}$, which will give the estimate for $T_1.$ First, 
consider $S_1$:
\[
\begin{aligned}
\left \vert S_1\right\vert & = \left\vert  \sum_{\alpha\in\Gamma_d} \sum_{I \in \mathcal{D}}  \sigma_{I,\alpha} C_I(w,\alpha) \left \langle 
w^{-\frac{1}{2}} \right \rangle_{E_{\alpha,I}} \widehat{\phi}(I,\alpha)\left \langle 
\psi w^{\frac{1}{2}}, h^{w,\alpha}_I \right \rangle_{L^2} \right\vert \\
& \le \| \sigma\|_{\infty}  \sum_{\alpha\in\Gamma_d} \sum_{I \in \mathcal{D}}  \left \vert
C_I(w,\alpha) \left \langle w^{-\frac{1}{2}} \right \rangle_{E_{\alpha,I}} \widehat{\phi}(I,\alpha)\left
\langle \psi w^{\frac{1}{2}}, h^{w,\alpha}_I \right \rangle_{L^2} \right\vert \\
& \lesssim  \| \sigma \|_{\infty} \| \phi\|_{L^2} \left( \sum_{\alpha\in\Gamma_d} \sum_{I \in \mathcal{D}}
\left \langle w \right \rangle_I \left \langle w^{-\frac{1}{2}} \right \rangle^2_{E_{\alpha,I}} 
\left \langle \psi w^{-\frac{1}{2}}, h^{w,\alpha}_I \right \rangle^2_{L^2(w)} 
\right)^{\frac{1}{2}}\\
& \lesssim  \| \sigma \|_{\infty} \left[ w\right]_{A_2}^{\frac{1}{2}}\| \phi\|_{L^2} \left( \sum_{\alpha\in\Gamma_d} \sum_{I \in \mathcal{D}} \left \langle \psi w^{-\frac{1}{2}}, h^{w,\alpha}_I \right \rangle^2_{L^2(w)} 
\right)^{\frac{1}{2}}\\
& \le  \| \sigma \|_{\infty} [w]^{\frac{1}{2}}_{A_2} \| \phi\|_{L^2} \left\| \psi w^{-\frac{1}{2}} \right\|_{L^2(w)}\\
& =\| \sigma \|_{\infty} [w]^{\frac{1}{2}}_{A_2} \| \phi\|_{L^2} \| \psi \|_{L^2}.
\end{aligned}
\]
Above we used $(\ref{C_Kest})$ coupled with the $A_2$ condition.  Lastly, we estimate $S_2$ as follows:
\[
\begin{aligned}
\left\vert S_2\right\vert & \le  \left\vert \sum_{\alpha\in\Gamma_d} \sum_{I \in
\mathcal{D}}  \sigma_{I,\alpha} D_I(w,\alpha) \left \langle w^{-\frac{1}{2}} \right 
\rangle_{E_{\alpha,I}} \widehat{\phi}(I,\alpha) \left \langle \psi w^{\frac{1}{2}}\right \rangle_{E_{\alpha,I}} \right\vert\\
& \leq \| \sigma \|_{\infty}   \sum_{\alpha\in\Gamma_d} \sum_{I \in \mathcal{D}}  \left \vert D_I(w,\alpha) 
\left \langle w^{-\frac{1}{2}} \right \rangle_{E_{\alpha,I}} \widehat{\phi}(I,\alpha)
\left \langle \psi w^{\frac{1}{2}}\right \rangle_{E_{\alpha,I}} \right \vert \\
& \le \| \sigma \|_{\infty} \| \phi \|_{L^2} \left( \sum_{\alpha\in\Gamma_d}
\sum_{I \in \mathcal{D}}  \frac{ \left\vert \widehat{w}(I,\alpha)\right\vert^2}{ \langle w \rangle_{E_{\alpha,I}}^2}  
\left \langle w^{-\frac{1}{2}} \right \rangle^2_{E_{\alpha,I}} \
\left \langle \psi w^{\frac{1}{2}}  \right \rangle^2_{E_{\alpha,I}} \right)^{\frac{1}{2}} \\
&\lesssim  \| \sigma \|_{\infty} [w]_{A_2} \| \phi \|_{L^2}  \left\| \psi \right\|_{L^2}, 
\end{aligned}
\]
where the third inequality follows via an application of the 
Carleson Embedding Theorem using estimate \eqref{SquareFunctionEst3}.
This establishes that each $|T_j| \lesssim \| \sigma \|_{\infty} [w]_{A_2}
\| \phi\|_{L^2} \| \psi \|_{L^2}$. Since $\phi,\psi \in L^2$ were arbitrary, we can use
duality to conclude
\[
\left \| P^{(0,1)}_{\widehat{w^{ \frac{1}{2}}}}
T_{\sigma} P^{(0,0)}_{\langle w^{- \frac{1}{2}} \rangle} \right \|_{L^2 
\rightarrow L^2} \lesssim \| \sigma \|_{\infty} [w]_{A_2}, 
\]
which finishes the proof of these terms.

\subsubsection{Estimating $P^{(0,1)}_{\widehat{w^{ \frac{1}{2}}}}T_{\sigma}P^{(1,0)}_{\widehat{w^{-\frac{1}{2}}}}$} \label{subsec:VeryDifficult}

To obtain estimates on the final term $P^{(0,1)}_{\widehat{w^{ \frac{1}{2}}}}
T_{\sigma}P^{(1,0)}_{\widehat{w^{- \frac{1}{2}}}}$ we simply use \eqref{canoncialpara} to observe that
\[
\begin{aligned} \left \| P^{(0,1)}_{\widehat{w^{ \frac{1}{2}}}}
T_{\sigma}P^{(1,0)}_{\widehat{w^{- \frac{1}{2}}}} \phi \right \|_{L^2} &\le \left \| M_{w^{\frac{1}{2}}} T_{\sigma} M_{w^{-\frac{1}{2}}} \phi  \right \|_{L^2} 
+  \left \| \mathsf{P}_{\widehat{w^{\frac{1}{2}}}}^{(1,0)}T_{\sigma}\mathsf{P}_{%
\widehat{w^{-\frac{1}{2}}}}^{(0,1)} \phi \right \|_{L^2} +  \left \|
\mathsf{P}_{\widehat{w^{\frac{1}{2}}}}^{(1,0)} T_{\sigma}\mathsf{P}%
_{\langle w^{-\frac{1}{2}}\rangle}^{(0,0)}\phi  \right \|_{L^2} \\
& \ \ + \left \| \mathsf{P}_{\langle w^{\frac{1}{2}}\rangle}^{(0,0)}T_{\sigma}\mathsf{P}_{%
\widehat{w^{-\frac{1}{2}}}}^{(0,1)} \phi  \right \|_{L^2} + 
\left \| \mathsf{P}_{\langle w^{\frac{1}{2}}\rangle}^{(0,0)}T_{\sigma}\mathsf{P}%
_{\langle w^{-\frac{1}{2}}\rangle}^{(0,0)} \phi  \right \|_{L^2}
+ \left \| \mathsf{P}_{\widehat{w^{\frac{1}{2}}}}^{(0,1)}T_{\sigma}\mathsf{P}_{%
\widehat{w^{-\frac{1}{2}}}}^{(0,1)} \phi  \right \|_{L^2} \\
& \ \ + \left \| \mathsf{P}_{\widehat{w^{\frac{1}{2}}}}^{(0,1)}T_{\sigma}\mathsf{P}_{\langle
w^{-\frac{1}{2}}\rangle}^{(0,0)} \phi   \right \|_{L^2}
+ \left \| \mathsf{P}_{\widehat{w^{\frac{1}{2}}}}^{(1,0)}T_{\sigma}\mathsf{P}_{%
\widehat{w^{-\frac{1}{2}}}}^{(1,0)} \phi \right \|_{L^2}
+ \left \| \mathsf{P}_{\langle w^{\frac{1}{2}}\rangle}^{(0,0)}T_{\sigma}\mathsf{P}_{%
\widehat{w^{-\frac{1}{2}}}}^{(1,0)} \phi \right \|_{L^2} \\
& \lesssim \left[ w \right]_{A_2} \| \phi \|_{L^2},
\end{aligned}
\]
using Theorem \ref{thm:haarbd} and our previous computations. This proof strategy motivates the open question:
\begin{question} Is there a proof of the linear bound for the final paraproduct composition
$P^{(0,1)}_{\widehat{w^{ \frac{1}{2}}}}T_{\sigma}P^{(1,0)}_{\widehat{w^{-\frac{1}{2}}}}$
that does not rely on the linear bound of $T_{\sigma}$ on $L^2(w)$?
\end{question}
To be precise, fix $\phi, \psi \in L^2.$  Then, the term of interest is
\[ 
\begin{aligned}
\left \langle P^{(0,1)}_{\widehat{w^{ \frac{1}{2}}}} T_{\sigma}
P^{(1,0)}_{\widehat{w^{- \frac{1}{2}}}} \phi, \psi \right \rangle_{L^2}
& = \left \langle P^{(0,1)}_{\widehat{w^{ \frac{1}{2}}}} 
T_{\sigma} \left( \sum_{\alpha\in\Gamma_d}\sum_{I \in \mathcal{D}} \widehat{w^{
-\frac{1}{2}}}(I,\alpha)\, \widehat{\phi}(I,\alpha) h^1_{E_{\alpha,I}}\right), \psi \right \rangle_{L^2}\\
&=  \sum_{\alpha,\beta\in\Gamma_d}\sum_{I,J \in \mathcal{D}} \widehat{w^{-\frac{1}{2}}}(I,\alpha) \,\widehat{\phi}(I,\alpha) \,\widehat{w^{ \frac{1}{2}}}(J,\beta)\, \widehat{\psi}(J,\beta) 
\left \langle T_{\sigma}h^1_{E_{\alpha, I}}, h^1_{E_{\beta, J}} \right \rangle_{L^2}\\
& = \sum_{\alpha,\beta,\gamma\in\Gamma_d}\sum_{I,J,K \in \mathcal{D}} \widehat{w^{-\frac{1}{2}}}(I,\alpha) \,\widehat{\phi}(I,\alpha) \,\widehat{w^{ \frac{1}{2}}}(J,\beta)\, \widehat{\psi}(J,\beta) \sigma_{K,\gamma} \, \widehat{h_{E_{\alpha, I}}^1}(K,\gamma) \widehat{h_{E_{\beta, J}}^1}(K,\gamma)\\
& = \sum_{\gamma\in\Gamma_d} \sum_{K \in \mathcal{D}} \sigma_{K, \gamma} \left(\sum_{\alpha \in \Gamma_d} \sum_{I \in \mathcal{D}} \widehat{w^{-\frac{1}{2}}}(I,\alpha) \,\widehat{\phi}(I,\alpha) \widehat{h_{E_{\alpha, I}}^1}(K,\gamma)  \right)\\
& \qquad \qquad \times \left( \sum_{\beta\in\Gamma_d}\sum_{J\in\mathcal{D}} \widehat{w^{ \frac{1}{2}}}(J,\beta)\, \widehat{\psi}(J,\beta) \widehat{h_{E_{\beta, J}}^1}(K,\gamma)\right).
\end{aligned}
\]
Currently, our tools seem unequal to the task of bounding this term without recourse to 
the bound for $T_{\sigma}.$ However, given such arguments, one would also obtain a new proof of the linear bound for $T_{\sigma}$ on $L^2(w)$.

\end{document}